\documentclass[11pt,reqno]{amsart}

\usepackage[text={160mm,240mm},centering]{geometry}            
\usepackage{amssymb,amsmath,setspace,geometry,indentfirst,changepage,inputenc,amsthm}
\usepackage{mathrsfs,amsfonts,savesym,graphicx,bm,color}

\usepackage{graphicx}
\usepackage{amssymb}
\usepackage{epstopdf}
\usepackage{dsfont}
\usepackage{placeins}
\usepackage{pifont}
\usepackage{romannum}

\usepackage{lscape}

\usepackage{tikz}
\usetikzlibrary{positioning}
\usepackage{xcolor}

\usepackage{float}

\usepackage[all,2cell]{xy}    
\UseAllTwocells               

\usepackage{amsmath,amsthm}

\usepackage[mathscr]{eucal}
\usepackage[all]{xy}
\usepackage{mathrsfs}
\usepackage{hyperref}
\usepackage{color}

\newtheorem{theorem}{Theorem}[section]
\newtheorem{lemma}[theorem]{Lemma}

\newtheorem*{main thm}{Main Theorem}

\newtheorem{conj}[theorem]{Conjecture}
\newtheorem{cor}[theorem]{Corollary}
\newtheorem{prop}[theorem]{Proposition}

\newtheorem{question}[theorem]{Question}
\newtheorem{def-lemma}[theorem]{Definition-Lemma}
\newtheorem{def-thm}[theorem]{Definition-Theorem}

\newtheorem{thm}[theorem]{Theorem}

\theoremstyle{definition}
\newtheorem{eg}[theorem]{Example}
\newtheorem{definition}[theorem]{Definition}

\newtheorem{remark}[theorem]{Remark}

\makeatletter
\AtBeginDocument{%
	\fontdimen16\textfont2 = 3.5pt
	\fontdimen17\textfont2 = 3.5pt
}
\makeatother

\allowdisplaybreaks

\title[]{On the Multiplicative Thom-Sebastiani Property for Bernstein-Sato Polynomials of Ideals}

\author{Yongxin Xu}
\address{
	Department of Mathematical Sciences,
	Tsinghua University,
	Beijing, 100084, P. R. China.}
\email{xuyongxi22@mails.tsinghua.edu.cn}

\author{Huaiqing Zuo}
\address{Department of Mathematical Sciences,
	Tsinghua University,
	Beijing, 100084, P. R. China.}
\email{hqzuo@mail.tsinghua.edu.cn}

\thanks{Corresponding author: Huaiqing Zuo}

\thanks{Data availability: Data sharing not applicable to this article as no datasets were generated or analysed during the current study}

\begin{document}

	\setcounter{page}{1}
	\pagenumbering{arabic}
	
	\maketitle
	
	\parskip=.3em
	
	\setcounter{tocdepth}{1}

	\begin{abstract}
		We prove that the Bernstein-Sato polynomials of ideals satisfy the multiplicative Thom-Sebastiani property. This affirmatively settles a question of Shi and Zuo that extends a question independently formulated by Budur and Popa.
		
		
	\end{abstract}
	
	\section{Introduction}\label{s1}
	
	\par Suppose $X=\mathbb{C}^{m}$ and let $\mathcal{O}_{X}(X)=\mathbb{C}[\bm{x}]=\mathbb{C}[x_{1},\cdots ,x_{m}]$ be the ring of regular functions on $X$. Throughout this paper, for a quasi-coherent sheaf $\mathcal{F}$ on an affine scheme, we abuse notation by identifying $\mathcal{F}$ with its module of global sections. Let $\mathcal{D}_{X}$ be the sheaf of differential operators on $X$. For $f\in \mathbb{C}[\bm{x}]\setminus \mathbb{C}$, the Bernstein-Sato polynomial $b_{f}(s)$ is defined as the monic polynomial of the smallest degree such that
	\begin{align*}
		b_{f}(s)f^{s}\in \mathcal{D}_{X}[s]\cdot f^{s+1},
	\end{align*} 
	where $s$ is a formal indeterminate. Its existence was established by Bernstein \cite{c1} in response to Gelfand's question on the meromorphic continuation of Archimedean zeta functions. A parallel theory was independently developed by Sato and his school in the context of prehomogeneous vector spaces, see \cite{c7}.
	
	\par The Bernstein-Sato polynomial is known to encode deep information about singularities. The numbers $\exp(2\pi \sqrt{-1}\alpha)$, as $\alpha$ ranges over the roots of $b_{f}(s)$, give the eigenvalues of the monodromy action on the cohomology of the Milnor fiber, see \cite{c9}, \cite{c10}, \cite{c11}. In a different direction, Kollár proved that the log canonical threshold of $f$ equals the negative of the largest root of $b_{f}(s)$, see \cite{c13}. The Bernstein-Sato polynomial thus forms a bridge between analytic continuation, the theory of D-modules, and birational invariants of singularities.
	
	\par For an ideal $I$ of $\mathbb{C}[\bm{x}]$, Budur, Mustaţă, and Saito defined the Bernstein-Sato polynomial $b_{I}(s)$ using the theory of the V-filtration, see \cite{c2}. Suppose $I$ is generated by $f_{1},\cdots ,f_{r}\in \mathbb{C}[\bm{x}]$. The Bernstein-Sato polynomial $b_{I}(s)$ is defined as the monic polynomial of the smallest degree such that
	\begin{align*}
		b_{I}(|\bm{s}|)f_{1}^{s_{1}}\cdots f_{r}^{s_{r}}\in \sum_{\bm{\alpha}\in \mathbb{Z}^{r},|\bm{\alpha}|=1}\mathcal{D}_{X}[\bm{s}]\cdot \bigg( \prod_{\alpha_{i}<0}\binom{s_{i}}{-\alpha_{i}}\bigg)f_{1}^{s_{1}+\alpha_{1}}\cdots f_{r}^{s_{r}+\alpha_{r}},
	\end{align*} 
	where $|\bm{s}|=s_{1}+\cdots +s_{r}$ and $|\bm{\alpha}|=\alpha_{1}+\cdots +\alpha_{r}$. Budur, Mustaţă, and Saito showed in \cite{c2} that $b_{I}(s)$ is independent of the choice of generators of $I$. This definition generalizes the classical one, since we have 
	\begin{align*}
		b_{f}(s)=b_{(f)}(s)
	\end{align*}
	for a principal ideal $(f)$. 
	
	\par An important conjecture in singularity theory is the monodromy conjecture. For simplicity, we only state its strong version in the Igusa zeta function setting. Let $p$ be a fixed prime, and let $I\subset \mathbb{Z}[x_{1},\cdots ,x_{m}]$ be an ideal. There exists a unique Haar measure $\mu_{p}$ on $\mathbb{Z}_{p}$ normalized so that $\mu_{p}(\mathbb{Z}_{p})=1$. We thus have $\mu_{p}(p^{k}\mathbb{Z}_{p})=\frac{1}{p^{k}}$ for $k\in \mathbb{Z}_{\ge 0}$. We denote the product measure on $(\mathbb{Z}_{p})^{m}$ by the same symbol $\mu_{p}$. Let $ord_{p}$ be the discrete valuation on $\mathbb{Z}_{p}$ with $ord_{p}(0)=+\infty$. We define
	\begin{align*}
		ord_{p}I(a)=\min\{ ord_{p}f(a)|f\in I\}
	\end{align*} 
	for $a\in (\mathbb{Z}_{p})^{m}$. The Igusa zeta function of $I$ is 
	\begin{align*}
	 	Z_{I,p}(s)=\int_{(\mathbb{Z}_{p})^{m}}\frac{d\mu_{p}}{|p^{ ord_{p}I(a)}|^{s}}
	\end{align*}
	for all $s\in \mathbb{C}$ with $Re(s)>0$. When $I$ is principal, it was shown by Igusa in \cite{c12} that $Z_{I,p}(s)$ is a rational function of $p^{-s}$. Hence it admits a meromorphic extension to $\mathbb{C}$. A similar argument works for general ideals, see \cite{c15} or \cite{c16}. The strong monodromy conjecture for a polynomial $f\in \mathbb{Z}[\bm{x}]$ predicts that if $s_{0}$ is a pole of $Z_{f,p}(s)$ of multiplicity $k$ for every sufficiently large prime $p$, then $s_{0}$ is a root of $b_{f}(s)$ of multiplicity at least $k$, see \cite{c17}. A natural generalization of this conjecture to ideals is the following statement:
	
	\begin{conj}\label{mc}
	    Let $I\subset \mathbb{Z}[x_{1},\cdots ,x_{m}]$ be an ideal and let $p$ be a sufficiently large prime. Suppose $s_{0}$ is a pole of $Z_{I,p}(s)$ with multiplicity $k$. Then $s_{0}$ is a root of $b_{I}(s)$ with multiplicity at least $k$.
	\end{conj}
	 
	\par It is straightforward to see that the Igusa zeta function satisfies the multiplicative Thom-Sebastiani property, that is, for ideals $I\subset Z[x_{1},\cdots ,x_{m}]$ and $J\subset Z[y_{1},\cdots ,y_{n}]$, we have 
	\begin{align*}
		Z_{IJ,p}(s)=Z_{I,p}(s)Z_{J,p}(s).
	\end{align*}
	The question of whether the Bernstein-Sato polynomial of a single polynomial satisfies the multiplicative Thom-Sebastiani property was raised by Budur in \cite{c3}, and by Popa in \cite{c4}. Shi and Zuo in \cite{c5} answered the question affirmatively. Lee independently proved the same affirmative result in \cite{c14}, using a different approach.
	
	\begin{thm}\label{c1.1}\textnormal{(\cite{c5},\cite{c14})}
		For $f\in \mathbb{C}[\bm{x}]=\mathbb{C}[x_{1},\cdots ,x_{m}]$ and $g\in \mathbb{C}[\bm{y}]=\mathbb{C}[y_{1},\cdots ,y_{n}]$, we have
		\begin{align*}
			b_{fg}(s)=b_{f}(s)b_{g}(s).
		\end{align*}
	\end{thm}
	
	\par It is natural to ask whether Bernstein-Sato polynomials of ideals satisfy the multiplicative Thom-Sebastiani property. Shi and Zuo proposed the question whether the question raised by Budur and Popa is valid in the ideal setting.
	
	\begin{question}\label{q1.1}\textnormal{(Question 3.5 in \cite{c5})}
		For non-zero ideals $I\subset \mathbb{C}[\bm{x}]=\mathbb{C}[x_{1},\cdots ,x_{m}]$ and $J\subset \mathbb{C}[\bm{y}]=\mathbb{C}[y_{1},\cdots ,y_{n}]$, is it true that $b_{IJ}(s)=b_{I}(s)b_{J}(s)$, where $IJ\subset \mathbb{C}[\bm{x},\bm{y}]$ is the ideal generated by $\{fg|f\in I,g\in J\}$?
	\end{question}
	
	\par In \cite{c5}, Shi and Zuo discussed this question in some special cases. They proved that the answer to this question is affirmative when $I$ or $J$ is a principal ideal. For monomial ideals $I$ and $J$, they showed that the root set of $b_{IJ}(s)$ equals the root set of $b_{I}(s)b_{J}(s)$.
	
	
	\par In this paper, we establish the following result, answering Question \ref{q1.1} in full generality.
	
	\begin{thm}\label{t1.1}
		Let $I\subset \mathbb{C}[\bm{x}]$ and $J\subset \mathbb{C}[\bm{y}]$ be two non-zero ideals. Then we have
		\begin{align*}
			b_{IJ}(s)=b_{I}(s)b_{J}(s).
		\end{align*} 
	\end{thm}
	
	\begin{remark}\label{r1.5}
		Compared with the single polynomial case, the Bernstein–Sato polynomial of an ideal involves more formal indeterminates. Motivated by the method in \cite{c5}, we obtain a useful characterization of $b_{I}(s)b_{J}(s)$, which we call the product-type weak Bernstein-Sato polynomial and denote by $b_{I, J}(s)$. However, it is not immediately obvious that $b_{I, J}(s)=b_{IJ}(s)$, because the presence of more formal indeterminates introduce new technical difficulties. We overcome these difficulties view a new approach. The desired morphism between the two functional equations  does not exist in general. We show that it does exist
		after inverting a suitable family of polynomials in $s$.
	\end{remark}
	
	\par Combining the multiplicative Thom-Sebastiani property for Igusa zeta functions and Bernstein-Sato polynomials, we obtain the following result.
	
	\begin{cor}\label{c4.5}
		The strong monodromy conjecture holds for $IJ$ if it holds for $I$ and $J$.
	\end{cor}
	
	\par This paper is organized as follows.
	
	\par In Section \ref{s2}, we provide preliminaries about Bernstein-Sato polynomials. In Section \ref{s3}, we define the product-type weak Bernstein-Sato polynomial $b_{I,J}(s)$ for two ideals $I\subset \mathbb{C}[\bm{x}]$ and $J\subset \mathbb{C}[\bm{y}]$. Then we prove that $b_{I}(s)b_{J}(s)=b_{I,J}(s)$. In Section \ref{s4}, we prove $
	b_{IJ}(s)=b_{I,J}(s)$, which completes the proof of Theorem \ref{t1.1}. In Section \ref{s5}, we give an alternative proof of the identity $b_{I}(s)b_{J}(s)=b_{I,J}(s)$ using results and methods from \cite{c6}.
			

	\section{Bernstein-Sato polynomials}\label{s2}
	
	\par In this section, we recall some basic results on the Bernstein-Sato polynomial.
	
	\par Bernstein-Sato polynomials are difficult to compute in general. Nevertheless, their roots are known to satisfy strong constraints. Kashiwara proved the following important result.
	
	\begin{thm}\label{t2.1}\textnormal{(\cite{c8})}
		Let $f\in \mathbb{C}[\bm{x}]\setminus \mathbb{C}$. Then the roots of $b_{f}(s)$ are negative rational numbers.
	\end{thm}
	
	\par For an ideal $I\subset \mathbb{C}[\bm{x}]$, the roots of $b_{I}(s)$ are negative rational numbers. This follows from the following result of Mustaţă, see also \cite{c2}.
	
	\begin{thm}\label{t2.2}\textnormal{(Theorem 1.1 in \cite{c6})}
		Let $I\subset \mathbb{C}[\bm{x}]$ be an ideal generated by $f_{1},\cdots ,f_{r}$. Suppose $F=f_{1}y_{1}+\cdots +f_{r}y_{r}$. Then we have
		\begin{align*}
			b_{F}(s)=(s+1)b_{I}(s).
		\end{align*}
	\end{thm}	
	
	\par The Bernstein-Sato polynomial $b_{I}(s)$ is unchanged under automorphisms of the ambient space.
	
	\begin{prop}\label{p2.5}
		Let $W$ be a $\mathbb{C}$-algebra automorphism of $\mathbb{C}[x_{1},\cdots ,x_{m}]$ (equivalently, an automorphism induced by a regular automorphism of $X=\mathbb{C}^{m}$). Then for every ideal $I\subset \mathbb{C}[x_{1},\cdots ,x_{m}]$, we have
		\begin{align*}
			b_{W(I)}(s)=b_{I}(s).
		\end{align*}
	\end{prop} 
	\begin{proof}[Proof. ]
		It suffices to show $b_{W(I)}(s)|b_{I}(s)$ since applying the same argument to $W^{-1}$ yields the reverse divisibility. Suppose $I$ is generated by $f_{1},\cdots ,f_{r}$, so that $W(I)=(W(f_{1}),\cdots ,W(f_{r}))$. The automorphism $W$ induces an automorphism
		\begin{align*}
			\widetilde{W}:\mathcal{D}_{X}\to \mathcal{D}_{X},P\mapsto W^{-1}\circ P\circ W.
		\end{align*}
		It extends to $\mathcal{D}_{X}[s_{1},\cdots ,s_{r}]$ by fixing $s_{1},\cdots ,s_{r}$. Suppose we have
		\begin{align*}
			b_{I}(|\bm{s}|)f_{1}^{s_{1}}\cdots f_{r}^{s_{r}}= \sum_{\bm{\alpha}\in \mathbb{Z}^{r},|\bm{\alpha}|=1}P_{\bm{\alpha}}(\bm{s})\cdot \bigg( \prod_{\alpha_{i}<0}\binom{s_{i}}{-\alpha_{i}}\bigg)f_{1}^{s_{1}+\alpha_{1}}\cdots f_{r}^{s_{r}+\alpha_{r}}.
		\end{align*}
		Applying $\widetilde{W}$ to both sides, we get
		\begin{align*}
		    b_{I}(|\bm{s}|)W(f_{1})^{s_{1}}\cdots W(f_{r})^{s_{r}}= \sum_{\bm{\alpha}\in \mathbb{Z}^{r},|\bm{\alpha}|=1}\widetilde{W}(P_{\bm{\alpha}}(\bm{s}))\cdot \bigg( \prod_{\alpha_{i}<0}\binom{s_{i}}{-\alpha_{i}}\bigg)W(f_{1})^{s_{1}+\alpha_{1}}\cdots W(f_{r})^{s_{r}+\alpha_{r}}.
		\end{align*}
		It follows that $b_{W(I)}(s)|b_{I}(s)$.
	\end{proof}
	
	\begin{eg}\label{e2.3}
		Let $I=(x_{1},\cdots ,x_{m})\subset \mathbb{C}[x_{1},\cdots ,x_{m}]$. The Bernstein-Sato polynomial is
		\begin{align*}
			b_{I}(s)=s+m,
		\end{align*}
		satisfying the functional equation
		\begin{align*}
			(s_{1}+\cdots +s_{m}+m)x_{1}^{s_{1}}\cdots x_{m}^{s_{m}}=\sum\limits_{i=1}^{m}\partial_{x_{i}}\prod\limits_{j=1}^{m}x_{j}^{s_{j}+\delta_{ij}},
		\end{align*}
		where $\delta_{ij}$ is the Kronecker delta..
	\end{eg}
	
	\begin{eg}\textnormal{(\cite{c2})}\label{e2.4}
		Let $I=(x_{1}x_{2},x_{2}x_{3},x_{1}x_{3})\subset \mathbb{C}[x_{1},x_{2} ,x_{3}]$. The Bernstein-Sato polynomial is
		\begin{align*}
			b_{I}(s)=(s+\frac{3}{2})(s+2)^{2}.
		\end{align*}
	\end{eg}
	
	\par Suppose $I$ is generated by $f_{1},\cdots ,f_{r}$. We give another characterization of $b_{I}(s)$, which will be useful later. For all $P\in \mathcal{D}_{X}[\bm{s}]$ and multi-indices $\bm{\alpha}\in \mathbb{Z}^{r}$, one can see that 
	\begin{align*}
		P\cdot f_{1}^{s_{1}+\alpha_{1}}\cdots f_{r}^{s_{r}+\alpha_{r}}\in \mathcal{O}_{X}[\frac{1}{f_{1}\cdots f_{r}}][\bm{s}]\cdot f_{1}^{s_{1}}\cdots f_{r}^{s_{r}}.
	\end{align*}
	We denote by $A_{X,I}$ the $\mathcal{O}_{X}[\bm{s}]$-module
	\begin{align*}
		\mathcal{O}_{X}[\frac{1}{f_{1}\cdots f_{r}}][\bm{s}],
	\end{align*}
	and by $M_{X,I}$ the $\mathcal{O}_{X}[\bm{s}]$-module
	\begin{align*}
		\{ F\in A_{X,I}|Ff_{1}^{s_{1}}\cdots f_{r}^{s_{r}}\in \sum_{\bm{\alpha}\in \mathbb{Z}^{r},|\bm{\alpha}|=1}\mathcal{D}_{X}[\bm{s}]\cdot \bigg( \prod_{\alpha_{i}<0}\binom{s_{i}}{-\alpha_{i}}\bigg)f_{1}^{s_{1}+\alpha_{1}}\cdots f_{r}^{s_{r}+\alpha_{r}}\}.
	\end{align*}
	We regard $\mathbb{C}[s]$ as a subring of $\mathbb{C}[\bm{s}]$ via the embedding $s\mapsto |\bm{s}|$. Then the Bernstein-Sato polynomial $b_{I}(s)$ is the unique monic generator of $M_{X,I}\cap \mathbb{C}[s]$. We denote by 
	\begin{align*}
		\mu_{X,I}:\mathbb{C}[s]\to A_{X,I}
	\end{align*}
	the natural embedding. One can restrict $\mu_{X,I}$ to the ideal $b_{I}(s)\cdot \mathbb{C}[s]\subset \mathbb{C}[s]$ to define 
	\begin{align*}
		\tau_{X,I}:b_{I}(s)\cdot \mathbb{C}[s]\to M_{X,I}.
	\end{align*}
	We have a Cartesian diagram of $\mathbb{C}[s]$-modules with all four arrows injective.
	\begin{equation*}
		\begin{tikzpicture}[baseline=(current bounding box.center)]
			\node(a) at (0,2){$b_{I}(s)\cdot \mathbb{C}[s]$};
			\node(b) at (3,2){$\mathbb{C}[s]$};
			\node(c) at (0,0){$M_{X,I}$};   
			\node(d) at (3,0){$A_{X,I}$};
			\draw[->] (a)--(b);
			\draw[->] (b)--(d) node[midway,right]{$\mu_{X,I}$};
			\draw[->] (a)--(c) node[midway,left]{$\tau_{X,I}$};
			\draw[->] (c)--(d);
		\end{tikzpicture}
	\end{equation*}

	\section{Product-type weak Bernstein-Sato polynomial}\label{s3}
	
	\par Let $X=\mathbb{C}^{m}$ and $Y=\mathbb{C}^{n}$. Let $I\subset \mathcal{O}_{X}=\mathbb{C}[\bm{x}]$ and $J\subset \mathcal{O}_{Y}=\mathbb{C}[\bm{y}]$ be two ideals. Suppose $I$ is generated by $f_{1},\cdots ,f_{r_{1}}$ and $J$ is generated by $g_{1},\cdots ,g_{r_{2}}$. We will define the product-type weak Bernstein-Sato polynomial $b_{I,J}(s)$ and show that $b_{IJ}(s)|b_{I,J}(s)$ and $b_{I,J}(s)=b_{I}(s)b_{J}(s)$ in this section. 
	
	\par We will use distinct formal indeterminates $u_{i}$, $v_{j}$, and $w_{ij}$ for the functional equations of $b_{I}(s)$, $b_{J}(s)$, and $b_{IJ}(s)$ respectively. To be explicit, we have 
	\begin{align*}
		& b_{I}(|\bm{u}|)f_{1}^{u_{1}}\cdots f_{r_{1}}^{u_{r_{1}}}\in \sum_{\bm{\alpha}\in \mathbb{Z}^{r_{1}},|\bm{\alpha}|=1}\mathcal{D}_{X}[\bm{u}]\cdot \bigg( \prod_{\alpha_{i}<0}\binom{u_{i}}{-\alpha_{i}}\bigg)f_{1}^{u_{1}+\alpha_{1}}\cdots f_{r_{1}}^{u_{r_{1}}+\alpha_{r_{1}}},\\
		& b_{J}(|\bm{v}|)g_{1}^{v_{1}}\cdots g_{r_{2}}^{v_{r_{2}}}\in \sum_{\bm{\beta}\in \mathbb{Z}^{r_{2}},|\bm{\beta}|=1}\mathcal{D}_{Y}[\bm{v}]\cdot \bigg( \prod_{\beta_{j}<0}\binom{v_{j}}{-\beta_{j}}\bigg)g_{1}^{v_{1}+\beta_{1}}\cdots g_{r_{2}}^{v_{r_{2}}+\beta_{r_{2}}},\\
		& b_{IJ}(|\bm{w}|)\prod_{i,j}(f_{i}g_{j})^{w_{ij}}\in \sum_{\bm{\gamma}\in \mathbb{Z}^{r_{1}r_{2}},|\bm{\gamma}|=1}\mathcal{D}_{X\times Y}[\bm{w}]\cdot \bigg( \prod_{\gamma_{ij}<0}\binom{w_{ij}}{-\gamma_{ij}}\bigg)\prod_{i,j}(f_{i}g_{j})^{w_{ij}+\gamma_{ij}}.
	\end{align*}
	The $\mathbb{C}[s]$-module structures on $\mathbb{C}[\bm{u}]$, $\mathbb{C}[\bm{v}]$, and $\mathbb{C}[\bm{w}]$ are defined respectively by
	\begin{center}
		$s\mapsto |\bm{u}|$,\\
		$s\mapsto |\bm{v}|$,\\
		$s\mapsto |\bm{w}|$.
	\end{center}	
	Intuitively, the relations among $u_{i}$, $v_{j}$, and $w_{ij}$ can be described as follows.
	\begin{align*}
		\begin{matrix}
			w_{11} & w_{12} & \cdots & w_{1r_{2}} & \to & u_{1}\\
			w_{21} & w_{22} & \cdots & w_{2r_{2}} & \to & u_{2}\\
			\vdots & \vdots & \ddots & \vdots & & \vdots \\
			w_{r_{1}1} & w_{r_{1}2} & \cdots & w_{r_{1}r_{2}} & \to & u_{r_{1}}\\
			\downarrow & \downarrow & & \downarrow & & \\
			v_{1} & v_{2} & \cdots & v_{r_{2}} & & 
		\end{matrix}
	\end{align*}
	
	\par This notation will be used throughout the rest of the paper.
	
	\subsection{Definition of product-type weak Bernstein-Sato polynomial.}
	
	\par $\quad$
	
	\par We first motivate the definition of the product-type weak Bernstein-Sato polynomial. For $I$ and $J$, we have two Cartesian diagrams.
	\begin{equation}\label{t1}
		\begin{tikzpicture}[baseline=(current bounding box.center)]
			\node(a) at (0,2){$b_{I}(s)\cdot \mathbb{C}[s]$};
			\node(b) at (3,2){$\mathbb{C}[s]$};
			\node(c) at (0,0){$M_{X,I}$};   
			\node(d) at (3,0){$A_{X,I}$};
			\draw[->] (a)--(b);
			\draw[->] (b)--(d) node[midway,right]{$\mu_{X,I}$};
			\draw[->] (a)--(c) node[midway,left]{$\tau_{X,I}$};
			\draw[->] (c)--(d);
		\end{tikzpicture}
	\end{equation}
	\begin{equation}\label{t2}
		\begin{tikzpicture}[baseline=(current bounding box.center)]
			\node(a) at (0,2){$b_{J}(s)\cdot \mathbb{C}[s]$};
			\node(b) at (3,2){$\mathbb{C}[s]$};
			\node(c) at (0,0){$M_{Y,J}$};   
			\node(d) at (3,0){$A_{Y,J}$};
			\draw[->] (a)--(b);
			\draw[->] (b)--(d) node[midway,right]{$\mu_{Y,J}$};
			\draw[->] (a)--(c) node[midway,left]{$\tau_{Y,J}$};
			\draw[->] (c)--(d);
		\end{tikzpicture}
	\end{equation}
	Tensoring \eqref{t1} and \eqref{t2} yields a commutative diagram of $\mathbb{C}[s]$-modules.
	\begin{equation}\label{t3}
		\begin{tikzpicture}[baseline=(current bounding box.center)]
			\node(a) at (0,2){$b_{I}(s)b_{J}(s)\cdot \mathbb{C}[s]$};
			\node(b) at (4,2){$\mathbb{C}[s]$};
			\node(c) at (0,0){$M_{X,I}\otimes_{\mathbb{C}[s]}M_{Y,J}$};   
			\node(d) at (4,0){$A_{X,I}\otimes_{\mathbb{C}[s]} A_{Y,J}$};
			\draw[->] (a)--(b);
			\draw[->] (b)--(d) node[midway,right]{$\mu_{X,I}\otimes \mu_{Y,J}$};
			\draw[->] (a)--(c) node[midway,left]{$\tau_{X,I}\otimes\tau_{Y,J}$};
			\draw[->] (c)--(d);
		\end{tikzpicture}
	\end{equation}
	
	\par The product-type weak Bernstein-Sato polynomial is designed to capture $(M_{X,I}\otimes_{\mathbb{C}[s]} M_{Y,J})\cap \mathbb{C}[s]$. For convenience, we use the notation $M_{I,J}=M_{X,I}\otimes_{\mathbb{C}[s]} M_{Y,J}$.
	
	\begin{definition}\label{d3.1}
		The product-type weak Bernstein-Sato polynomial $b_{I,J}(s)$ is the unique monic polynomial of the smallest degree such that $b_{I,J}(|\bm{u}|)f_{1}^{u_{1}}\cdots f_{r_{1}}^{u_{r_{1}}}g_{1}^{v_{1}}\cdots g_{r_{2}}^{v_{r_{2}}}$ is in
		\begin{align*}  
	        \sum_{|\bm{\alpha}|=|\bm{\beta}|=1}\frac{\mathcal{D}_{X\times Y}[\bm{u},\bm{v}]}{(|\bm{u}|-|\bm{v}|)}\cdot \bigg( \prod_{\alpha_{i}<0}\binom{u_{i}}{-\alpha_{i}}\prod_{\beta_{j}<0}\binom{v_{j}}{-\beta_{j}}\bigg)f_{1}^{u_{1}+\alpha_{1}}\cdots f_{r_{1}}^{u_{r_{1}}+\alpha_{r_{1}}}g_{1}^{v_{1}+\beta_{1}}\cdots g_{r_{2}}^{v_{r_{2}}+\beta_{r_{2}}}.
		\end{align*}
	\end{definition}
	
	\par One can see that $b_{I,J}(s)\cdot \mathbb{C}[s]=M_{X,I}\otimes M_{Y,J}\cap \mathbb{C}[s]$ and the following commutative diagram is Cartesian.
	
	\begin{equation}\label{t4}
		\begin{tikzpicture}[baseline=(current bounding box.center)]
			\node(a) at (0,2){$b_{I,J}(s)\cdot \mathbb{C}[s]$};
			\node(b) at (4,2){$\mathbb{C}[s]$};
			\node(c) at (0,0){$M_{I,J}$};   
			\node(d) at (4,0){$A_{X,I}\otimes_{\mathbb{C}[s]} A_{Y,J}$};
			\draw[->] (a)--(b);
			\draw[->] (b)--(d) node[midway,right]{$\mu_{X,I}\otimes \mu_{Y,J}$};
			\draw[->] (a)--(c) node[midway,left]{$\tau_{X,I}\otimes\tau_{Y,J}$};
			\draw[->] (c)--(d);
		\end{tikzpicture}
	\end{equation}
	
	\subsection{Relations among $b_{I}(s)b_{J}(s)$, $b_{I,J}(s)$, and $b_{IJ}(s)$.}
	
	\par $\quad$
	
	\par We will prove $b_{IJ}(s)|b_{I,J}(s)$ and $b_{I,J}(s)=b_{I}(s)b_{J}(s)$ in this subsection.
	
	\par We show $b_{IJ}(s)|b_{I,J}(s)$ first. The Cartesian diagram corresponding to $b_{IJ}(s)$ is as follows.
	\begin{equation}\label{t5}
		\begin{tikzpicture}[baseline=(current bounding box.center)]
			\node(a) at (0,2){$b_{IJ}(s)\cdot \mathbb{C}[s]$};
			\node(b) at (4,2){$\mathbb{C}[s]$};
			\node(c) at (0,0){$M_{X\times Y,IJ}$};   
			\node(d) at (4,0){$A_{X\times Y,IJ}$};
			\draw[->] (a)--(b);
			\draw[->] (b)--(d) node[midway,right]{$\mu_{X\times Y,IJ}$};
			\draw[->] (a)--(c) node[midway,left]{$\tau_{X\times Y,IJ}$};
			\draw[->] (c)--(d);
		\end{tikzpicture}
	\end{equation}
	Let $R_{1}=\mathbb{C}[\bm{u}]$, $R_{2}=\mathbb{C}[\bm{v}]$, and $S=\mathbb{C}[\bm{w}]$. Let $R=R_{1}\otimes_{\mathbb{C}[s]}R_{2}=\mathbb{C}[\bm{u},\bm{v}]/(|\bm{u}|-|\bm{v}|)$. We embed $R_{1}$ and $R_{2}$ into $S$ by
	\begin{align*}
		u_{i}\mapsto \sum\limits_{j=1}^{r_{2}}w_{ij}, v_{j}\mapsto \sum\limits_{i=1}^{r_{1}}w_{ij}
	\end{align*}
	respectively, which induces a natural embedding $R\to S$. We have 
	\begin{center}
		$A_{X,I}\otimes_{\mathbb{C}[s]} A_{Y,J}=\mathcal{O}_{X\times Y}[\frac{1}{f_{1}\cdots f_{r_{1}}g_{1}\cdots g_{r_{2}}}]\otimes R$,\\
		$A_{X\times Y,IJ}=\mathcal{O}_{X\times Y}[\frac{1}{f_{1}\cdots f_{r_{1}}g_{1}\cdots g_{r_{2}}}]\otimes S$,
	\end{center}
	so the embedding $R\to S$ induces an embedding $T:A_{X,I}\otimes_{\mathbb{C}[s]} A_{Y,J}\to A_{X\times Y,IJ}$.
	
	\begin{lemma}\label{l3.2}
		We have $T(M_{I,J})\subset M_{X\times Y,IJ}$.
	\end{lemma}
	\begin{proof}[Proof. ]
		Let $F\in M_{X,I}$ and $G\in M_{Y,J}$ satisfy 
		\begin{align*}
			Ff_{1}^{u_{1}}\cdots f_{r_{1}}^{u_{r_{1}}} & =P(\bm{u})\cdot \bigg( \prod_{\alpha_{i}<0}\binom{u_{i}}{-\alpha_{i}}\bigg)f_{1}^{u_{1}+\alpha_{1}}\cdots f_{r_{1}}^{u_{r_{1}}+\alpha_{r_{1}}},\\
			Gg_{1}^{v_{1}}\cdots g_{r_{2}}^{v_{r_{2}}} & =Q(\bm{v})\cdot \bigg( \prod_{\beta_{j}<0}\binom{v_{j}}{-\beta_{j}}\bigg)g_{1}^{v_{1}+\beta_{1}}\cdots g_{r_{2}}^{v_{r_{2}}+\beta_{r_{2}}}
		\end{align*}
		for some $\bm{\alpha}\in \mathbb{Z}^{r_{1}}$ and $\bm{\beta}\in \mathbb{Z}^{r_{2}}$ satisfying $|\bm{\alpha}|=|\bm{\beta}|=1$. It suffices to show that $T(F\otimes G)\in M_{X\times Y,IJ}$, that is, 
		\begin{align*}
			T(F\otimes G)\prod_{i,j}(f_{i}g_{j})^{w_{ij}}\in \sum_{\substack{\bm{\gamma}\in \mathbb{Z}^{r_{1}r_{2}}\\
			|\bm{\gamma}|=1}}\mathcal{D}_{X\times Y}[\bm{w}]\cdot \bigg( \prod_{\gamma_{ij}<0}\binom{w_{ij}}{-\gamma_{ij}}\bigg)\prod_{i,j}(f_{i}g_{j})^{w_{ij}+\gamma_{ij}}.
		\end{align*}
		The product $\prod\limits_{i,j}(f_{i}g_{j})^{w_{ij}}$ is equal to $f_{1}^{u_{1}}\cdots f_{r_{1}}^{u_{r_{1}}}g_{1}^{v_{1}}\cdots g_{r_{2}}^{v_{r_{2}}}$ if we regard $R_{1}$ and $R_{2}$ as subrings of $S$. It follows that
		\begin{align*}
			& T(F\otimes G)\prod_{i,j}(f_{i}g_{j})^{w_{ij}}\\
			= & FGf_{1}^{u_{1}}\cdots f_{r_{1}}^{u_{r_{1}}}g_{1}^{v_{1}}\cdots g_{r_{2}}^{v_{r_{2}}}\\
			= & P(\bm{u})Q(\bm{v})\cdot \bigg( \prod_{\alpha_{i}<0}\binom{u_{i}}{-\alpha_{i}} \prod_{\beta_{j}<0}\binom{v_{j}}{-\beta_{j}}\bigg)f_{1}^{u_{1}+\alpha_{1}}\cdots f_{r_{1}}^{u_{r_{1}}+\alpha_{r_{1}}}g_{1}^{v_{1}+\beta_{1}}\cdots g_{r_{2}}^{v_{r_{2}}+\beta_{r_{2}}}.
		\end{align*}
		\par For all $r,k\in \mathbb{Z}_{\ge 0}$, let $T_{r,k}=\{ (k_{1},\cdots ,k_{r})\in \mathbb{Z}^{r}_{\ge 0}|k_{1}+\cdots +k_{r}=k\}$. Using the identity
		\begin{align*}
			\binom{z_{1}+\cdots +z_{r}}{k}=\sum_{(k_{1},\cdots ,k_{r})\in T_{r,k}}\binom{z_{1}}{k_{1}}\cdots \binom{z_{r}}{k_{r}},
		\end{align*}
		we have
		\begin{align*}
			& \prod_{\alpha_{i}<0}\binom{u_{i}}{-\alpha_{i}} \prod_{\beta_{j}<0}\binom{v_{j}}{-\beta_{j}}\\
			= & \prod_{\alpha_{i}<0}\sum_{(c_{i1},\cdots ,c_{ir_{2}})\in T_{r_{2},-\alpha_{i}}}\binom{w_{i1}}{c_{i1}}\cdots \binom{w_{ir_{2}}}{c_{ir_{2}}} \prod_{\beta_{j}<0}\sum_{(d_{1j},\cdots ,d_{r_{1}j})\in T_{r_{1},-\beta_{j}}}\binom{w_{1j}}{d_{1j}}\cdots \binom{w_{r_{1}j}}{d_{r_{1}j}}\\
			= & \sum_{\bm{c}\in T^{1}_{\bm{\alpha}}}\prod\limits_{c_{ij}<0}\binom{w_{ij}}{-c_{ij}}\sum_{\bm{d}\in T^{2}_{\bm{\beta}}}\prod\limits_{d_{ij}<0}\binom{w_{ij}}{-d_{ij}},
		\end{align*}
		where the sets $T^{1}_{\bm{\alpha}}$ and $T^{2}_{\bm{\beta}}$ are 
		\begin{align*}
			& T^{1}_{\bm{\alpha}}=\{ \bm{c}\in \mathbb{Z}^{r_{1}r_{2}}\;|\; \text{$\alpha_{i}<0$ $\Rightarrow$ $c_{ij}\le 0$, $\alpha_{i}\ge 0$ $\Rightarrow$ $c_{ij}\ge 0$, and $\sum\limits_{k=1}^{r_{2}}c_{ik}=\alpha_{i}$ for all $i,j$}\},\\
			& T^{2}_{\bm{\beta}}=\{ \bm{d}\in \mathbb{Z}^{r_{1}r_{2}}\;|\; \text{$\beta_{j}<0$ $\Rightarrow$ $d_{ij}\le 0$, $\beta_{j}\ge 0$ $\Rightarrow$ $d_{ij}\ge 0$, and $\sum\limits_{k=1}^{r_{1}}d_{kj}=\beta_{j}$ for all $i,j$}\}.
		\end{align*}
		\par For all $a,b\in \mathbb{R}$, let $A(a,b)=\min \{ a,b,0\}$. Then we have
		\begin{align*}
			& \prod_{\alpha_{i}<0}\binom{u_{i}}{-\alpha_{i}}\prod_{\beta_{j}<0}\binom{v_{j}}{-\beta_{j}}\\
			= & \sum_{\bm{c}\in T^{1}_{\bm{\alpha}},\bm{d}\in T^{2}_{\bm{\beta}}}\prod\limits_{c_{ij}<0}\binom{w_{ij}}{-c_{ij}}\prod\limits_{d_{ij}<0}\binom{w_{ij}}{-d_{ij}}\\
			\in & \sum_{\bm{c}\in T^{1}_{\bm{\alpha}},\bm{d}\in T^{2}_{\bm{\beta}}}\prod\limits_{i,j}\binom{w_{ij}}{-A(c_{ij},d_{ij})}\cdot S.
		\end{align*}
		\par Let $A_{ij}=A(c_{ij},d_{ij})$. It remains to show that 
		\begin{align*}
			f_{1}^{\alpha_{1}}\cdots f_{r_{1}}^{\alpha_{r_{1}}}g_{1}^{\beta_{1}}\cdots g_{r_{2}}^{\beta_{r_{2}}}\bigg( \prod\limits_{i,j}\binom{w_{ij}}{-A_{ij}}\bigg)\prod\limits_{i,j}(f_{i}g_{j})^{w_{ij}}
		\end{align*}
		is in $\sum\limits_{|\bm{\gamma}|=1}D_{X\times Y}[\bm{w}]\big( \prod\limits_{\gamma_{ij}<0}\binom{w_{ij}}{-\gamma_{ij}}\big)\prod\limits_{i,j}(f_{i}g_{j})^{w_{ij}+\gamma_{ij}}$. For all $i,j$, one can see that
		\begin{align*}
			& \alpha_{i}-A_{i1}-\cdots -A_{ir_{2}}\ge \alpha_{i}-c_{i1}-\cdots -c_{ir_{2}}=0,\\
			& \beta_{j}-A_{1j}-\cdots -A_{r_{1}j}\ge \beta_{j}-d_{1j}-\cdots -d_{r_{1}j}\ge 0.
		\end{align*}
		Hence there exists $\bm{k}\in \mathbb{Z}_{\ge 0}^{r_{1}r_{2}}$ such that
		\begin{align*}
			\prod_{i=1}^{r_{1}}f_{i}^{\alpha_{i}-A_{i1}-\cdots -A_{ir_{2}}}\prod_{j=1}^{r_{2}}g_{j}^{\beta_{j}-A_{1j}-\cdots -A_{r_{1}j}}=\prod\limits_{i,j}(f_{i}g_{j})^{k_{ij}}.
		\end{align*}
		So we have 
		\begin{align*}
			& f_{1}^{\alpha_{1}}\cdots f_{r_{1}}^{\alpha_{r_{1}}}g_{1}^{\beta_{1}}\cdots g_{r_{2}}^{\beta_{r_{2}}}\bigg( \prod\limits_{i,j}\binom{w_{ij}}{-A_{ij}}\bigg)\prod\limits_{i,j}(f_{i}g_{j})^{w_{ij}}\\
			= & \prod_{i=1}^{r_{1}}f_{i}^{\alpha_{i}-A_{i1}-\cdots -A_{ir_{2}}}\prod_{j=1}^{r_{2}}g_{j}^{\beta_{j}-A_{1j}-\cdots -A_{r_{1}j}}
			\binom{w_{ij}}{-A_{ij}}\prod\limits_{i,j}(f_{i}g_{j})^{w_{ij}+A_{ij}}\\
			= & \prod\limits_{i,j}\binom{w_{ij}}{-A_{ij}}\prod\limits_{i,j}(f_{i}g_{j})^{w_{ij}+A_{ij}+k_{ij}}\\
			\in & \sum\limits_{|\bm{\gamma}|=1}D_{X\times Y}[\bm{w}]\bigg( \prod\limits_{\gamma_{ij}<0}\binom{w_{ij}}{-\gamma_{ij}}\bigg)\prod\limits_{i,j}(f_{i}g_{j})^{w_{ij}+\gamma_{ij}}
		\end{align*}
		since $\sum\limits_{i,j}(A_{ij}+k_{ij})=\frac{1}{2}(|\bm{\alpha}|+|\bm{\beta}|)=1$.
	\end{proof}
	
	\begin{prop}\label{p3.3}
		We have $b_{IJ}(s)|b_{I,J}(s)$.
	\end{prop}
	\begin{proof}[Proof. ]
		For the two Cartesian diagrams \eqref{t4} and \eqref{t5}, we have $\mathbb{C}[s]$-module morphisms
		\begin{align*}
			& T:A_{X,I}\otimes_{\mathbb{C}[s]} A_{Y,J}\to A_{X\times Y,IJ},\\
			& T:M_{X,I}\otimes_{\mathbb{C}[s]} M_{Y,J}\to M_{X\times Y,IJ},\\
			& id_{\mathbb{C}[s]}:\mathbb{C}[s]\to \mathbb{C}[s],
		\end{align*} 
		which are all injective. 
		\begin{figure}[h]
			\centering
			\begin{tikzpicture}
				\node(a) at (0,3){$b_{I,J}(s)\cdot \mathbb{C}[s]$};
				\node(b) at (5.3,3){$\mathbb{C}[s]$};
				\node(c) at (0,0){$M_{X,I}\otimes_{\mathbb{C}[s]}M_{Y,J}$};   
				\node(d) at (5.3,0){$A_{X,I}\otimes_{\mathbb{C}[s]} A_{Y,J}$};
				\draw[->] (a)--(b);
				\draw[->] (b)--(d);
				\draw[->] (a)--(c);
				\draw[->] (c)--(d);
				\node(a1) at (2.7,4){$b_{IJ}(s)\cdot \mathbb{C}[s]$};
				\node(b1) at (8,4){$\mathbb{C}[s]$};
				\node(c1) at (2.7,1){$M_{X\times Y,IJ}$};   
				\node(d1) at (8,1){$A_{X\times Y,IJ}$};
				\draw[->] (a1)--(b1);
				\draw[->] (b1)--(d1);
				\draw[->] (a1)--(c1);
				\draw[->] (c1)--(d1);
				\draw[->] (b)--(b1) node[midway,above]{$id$};
				\draw[->] (c)--(c1) node[midway,above]{$T$};
				\draw[->] (d)--(d1) node[midway,above]{$T$};
				\draw[dashed][->] (a)--(a1);
			\end{tikzpicture}
		\end{figure}	\FloatBarrier
		
		\par \noindent By the universal property of the Cartesian diagram, $b_{I,J}(s)\cdot \mathbb{C}[s]$ can be canonically embedded into $b_{IJ}(s)\cdot \mathbb{C}[s]$. Hence we have $b_{IJ}(s)|b_{I,J}(s)$.
	\end{proof}
	
	\par Our next goal is to show that $b_{I,J}(s)=b_{I}(s)b_{J}(s)$. As a preparation, we introduce filtrations on $M_{X,I}$ and $M_{Y,J}$. For all $p\in \mathbb{Z}$, let $F_{p}\mathcal{D}_{X}$ be the $\mathcal{O}_{X}$-module generated by all differential operators of order at most $p$. We define
	\begin{center}
		$N_{X,I,p}=\sum\limits_{\substack{|\bm{\alpha}|=1\\
		|\alpha_{1}|+\cdots +|\alpha_{r_{1}}|\le p}}F_{p}\mathcal{D}_{X}[\bm{u}]\cdot \bigg( \prod_{\alpha_{i}<0}\binom{u_{i}}{-\alpha_{i}}\bigg)f_{1}^{u_{1}+\alpha_{1}}\cdots f_{r_{1}}^{u_{r_{1}}+\alpha_{r_{1}}}$,\\
		$N_{Y,J,p}=\sum\limits_{\substack{|\bm{\beta}|=1\\
		|\beta_{1}|+\cdots +|\beta_{r_{2}}|\le p}}F_{p}\mathcal{D}_{Y}[\bm{v}]\cdot \bigg( \prod_{\beta_{j}<0}\binom{v_{j}}{-\beta_{j}}\bigg)g_{1}^{v_{1}+\beta_{1}}\cdots g_{r_{2}}^{v_{r_{2}}+\beta_{r_{2}}}$.
	\end{center}
	The filtrations on $M_{X,I}$ and $M_{Y,J}$ are defined as
	\begin{center}
		$F_{p}M_{X,I}=\{ F\in A_{X,I}|Ff_{1}^{u_{1}}\cdots f_{r_{1}}^{u_{r_{1}}}\in N_{X,I,p}\}$,\\
		$F_{p}M_{Y,J}=\{ G\in A_{Y,J}|Gg_{1}^{v_{1}}\cdots g_{r_{2}}^{v_{r_{2}}}\in N_{Y,J,p}\}$.
	\end{center}
	
	\begin{lemma}\label{l3.4}\textnormal{(Lemma 3.1 in \cite{c5})}
	    Let $A$ be a PID, and let $M$ and $N$ be free $A$-modules. Suppose $M_{1},M_{2}\subset M$ and $N_{1},N_{2}\subset N$ are submodules. Then the identity
	    \begin{align*}
	    	(M_{1}\cap M_{2})\otimes (N_{1}\cap N_{2})=\bigcap\limits_{i,j\in \{ 1,2\}}M_{i}\otimes N_{j}
	    \end{align*} 
	    holds.
	\end{lemma}
	
	\begin{prop}\label{p3.5}
		The equality $b_{I,J}(s)=b_{I}(s)b_{J}(s)$ holds.
	\end{prop}
	\begin{proof}[Proof. ]
		The divisibility $b_{I,J}(s)|b_{I}(s)b_{J}(s)$ is trivial. To prove the divisibility $b_{I}(s)b_{J}(s)|b_{I,J}(s)$, it suffices to show that \eqref{t3} is Cartesian. The $\mathbb{C}[s]$-modules $A_{X,I}$ and $A_{Y,J}$ are free. By Lemma \ref{l3.4}, we have $b_{I}(s)b_{J}(s)\cdot \mathbb{C}[s]=M_{1}\cap M_{2}\cap M_{3}\cap M_{4}$, where
		\begin{center}
			$M_{1}=M_{X,I}\otimes_{\mathbb{C}[s]} M_{Y,J}$,\\
			$M_{2}=M_{X,I}\otimes_{\mathbb{C}[s]} \mathbb{C}[s]$,\\
			$M_{3}=\mathbb{C}[s]\otimes_{\mathbb{C}[s]} M_{Y,J}$,\\
			$M_{4}=\mathbb{C}[s]\otimes_{\mathbb{C}[s]} \mathbb{C}[s]$.
		\end{center}
		The $\mathbb{C}[s]$-module $M_{I,J}\cap \mathbb{C}[s]$ equals $M_{1}\cap M_{4}$. Hence it remains to show that $M_{1}\cap M_{4}\subset M_{2}$ and $M_{1}\cap M_{4}\subset M_{3}$. We only prove $M_{1}\cap M_{4}\subset M_{2}$ since $M_{1}\cap M_{4}\subset M_{3}$ can be proved by a similar argument. 
		\par The $\mathbb{C}[s]$-module morphism $T:A_{X,I}\otimes A_{Y,J}\to A_{X\times Y,IJ}$ is injective, so it suffices to show $T(M_{1}\cap M_{4})\subset T(M_{2})$. Using the filtration, we only need to prove
		\begin{align*}
			T(F_{r_{2}p}M_{X,I}\otimes F_{r_{1}p}M_{Y,J}\cap \mathbb{C}[s])\subset T(F_{r_{2}p}M_{X,I})\otimes \mathbb{C}[s]
		\end{align*}
		for all $p\in \mathbb{Z}_{\ge 0}$. Let $B\in F_{r_{2}p}M_{X,I}\otimes F_{r_{1}p}M_{Y,J}$. We write $B$ as
		\begin{align*}
			B=\sum\limits_{i=1}^{c}F_{i}(\bm{u})f_{1}^{-2r_{2}p}\cdots f_{r_{1}}^{-2r_{2}p}\otimes G_{i}(\bm{v})g_{1}^{-2r_{1}p}\cdots g_{r_{2}}^{-2r_{1}p},
		\end{align*}
		where $F_{i}(\bm{u})\in \mathcal{O}_{X}[\bm{u}]$ and $G_{i}(\bm{v})\in \mathcal{O}_{Y}[\bm{v}]$. Suppose $B=c(|\bm{u}|)\in \mathbb{C}[\bm{u}]$ for some polynomial $c(s)\in \mathbb{C}[s]$. Then we have 
		\begin{align*}
			\sum\limits_{i=1}^{c}F_{i}(\bm{u})\otimes G_{i}(\bm{v})=c(|\bm{w}|)f_{1}^{2r_{2}p}\cdots f_{r_{1}}^{2r_{2}p}g_{1}^{2r_{1}p}\cdots g_{r_{2}}^{2r_{1}p}.
		\end{align*}
		It remains to show $c(|\bm{w}|)\in T(F_{r_{2}p}M_{X,I})$ since $T(B)=c(|\bm{w}|)\otimes 1$ and $1\in \mathbb{C}[s]$.
		\par By Proposition \ref{p2.5}, we can suppose that $g_{1},\cdots ,g_{r_{2}}\in \mathfrak{m}=(y_{1},\cdots ,y_{n})$ without loss of generality. There exists $k\in \mathbb{Z}_{\ge 0}$ such that $g_{1}\cdots g_{r_{2}}\in \mathfrak{m}^{k}\setminus \mathfrak{m}^{k+1}$. For the set 
		\begin{align*}
			Z_{k}=\{ \bm{\alpha}\in \mathbb{Z}_{\ge 0}^{n}||\bm{\alpha}|=k\},
		\end{align*}
		define an order $\le$ on it as follows: $\bm{\alpha}< \bm{\beta}$ if and only if there exists an integer $l$ with $0\le l<n$ such that $\alpha_{1}=\beta_{1}$, $\cdots $, $\alpha_{l}=\beta_{l}$, and $\alpha_{l+1}<\beta_{l+1}$. Choose $\bm{\alpha}\in Z_{k}$ to be the minimal element with respect to this order for which the coefficient of $\bm{y}^{\bm{\alpha}}$ in $g_{1}\cdots g_{r_{2}}$ being non-zero. We can suppose this coefficient is $1$ without loss of generality. The monomials also form a $\mathbb{C}[\bm{v}]$-linear basis of $\mathcal{O}_{Y}[\bm{v}]$. For all $i$, writing $G_{i}(\bm{v})$ in the monomial basis, we suppose the coefficient of $\bm{y}^{2r_{1}p\bm{\alpha}}$ in $G_{i}(\bm{v})$ is $a_{i}(\bm{v})\in S=\mathbb{C}[\bm{w}]$. Then we have 
		\begin{equation}\label{e6}
			\sum\limits_{i=1}^{c}a_{i}(\bm{v})F_{i}(\bm{u}) =c(|\bm{w}|)f_{1}^{2r_{2}p}\cdots f_{r_{1}}^{2r_{2}p}
		\end{equation}
		by comparing the coefficient of $\bm{y}^{2r_{1}p\alpha}$. The polynomial ring $S$ is equal to
		\begin{align*}
			\mathbb{C}[\bm{u}][w_{12},\cdots ,w_{1r_{2}},\cdots ,w_{r_{1}2},\cdots ,w_{r_{1}r_{2}}],
		\end{align*}
		so $a_{i}(\bm{v})$ can be written in the form
		\begin{align*}
			a_{i}(\bm{v})=a_{i}(u_{1}-w_{12}-\cdots -w_{1r_{2}}+\cdots +u_{r_{1}}-w_{r_{1}2}-\cdots -w_{r_{1}r_{2}},v_{2},\cdots ,v_{r_{2}}).
		\end{align*}
		Let $\pi:S\to \mathbb{C}[\bm{u}]$ be the natural projection induced by taking $w_{12}=\cdots =w_{1r_{2}}=\cdots =w_{r_{1}2}=\cdots =w_{r_{1}r_{2}}=0$ with the identity $S=\mathbb{C}[\bm{u}][w_{12},\cdots ,w_{1r_{2}},\cdots ,w_{r_{1}2},\cdots ,w_{r_{1}r_{2}}]$. We have $\pi(a_{i}(\bm{v}))=a_{i}(|\bm{u}|,0,\cdots ,0)$ and $\pi(c(|\bm{w}|))=c(|\bm{w}|)$. By the equality \eqref{e6}, we have
		\begin{align*}
			\sum\limits_{i=1}^{c}a_{i}(|\bm{u}|,0,\cdots ,0)F_{i}(\bm{u}) =c(|\bm{w}|)f_{1}^{2r_{2}p}\cdots f_{r_{1}}^{2r_{2}p},
		\end{align*}
		so $c(|\bm{w}|)\in T(F_{r_{2}p}M_{X,I})$.
	\end{proof}

	\section{Proof of $b_{IJ}(s)=b_{I,J}(s)$}\label{s4}
	
	\par We adopt the notation from the previous sections. We have proved $b_{I,J}(s)=b_{I}(s)b_{J}(s)$ and $b_{IJ}(s)|b_{I,J}(s)$. It remains to show $b_{I,J}(s)|b_{IJ}(s)$ for the proof of Theorem \ref{t1.1}. We have two Cartesian diagrams \eqref{t4} and $\eqref{t5}$. The natural idea is to construct a $\mathbb{C}[s]$-module morphism from $M_{X\times Y,IJ}$ to $M_{I,J}$, as in the proof of Proposition \ref{p3.3}. The essential step is to construct a suitable $\mathbb{C}[s]$-module morphism from $S$ to $R$. Unfortunately, such a morphism is difficult to construct directly. However, we can construct an $R$-module morphism $\sigma:S\to L^{-1}R$ that suffices for our purpose, where $L\subset \mathbb{C}[s]$ is the multiplicatively closed subset generated by $\{ s-N|N\in \mathbb{Z}_{\ge 0}\}$. We will use this $\mathbb{C}[s]$-module morphism to prove $b_{I,J}(s)|b_{IJ}(s)$ in this section.
	
	\par To construct $\sigma:S\to L^{-1}R$, we introduce some new notation. For all $a\in \mathbb{Z}_{\ge 0}$, let $(z)^{\underline{a}}=a!\binom{z}{a}$, where $z$ is an indeterminate. In particular, $(z)^{\underline{0}}=1$. For a matrix-valued multi-index $K=(k_{ij})\in \mathbb{Z}_{\ge 0}^{r_{1}r_{2}}$, we introduce the notations
	\begin{center}
		$k_{i-}=k_{i1}+\cdots +k_{ir_{2}}$,\\
		$k_{-j}=k_{1j}+\cdots +k_{r_{1}j}$,\\
		$|K|=\sum_{i,j}k_{ij}$,\\
		$F_{K}(\bm{w})=\prod_{i,j}(w_{ij})^{\underline{k_{ij}}}\in S$.
	\end{center} 
	Intuitively, these notations can be described as follows.
	\begin{align*}
		\begin{matrix}
			k_{11} & k_{12} & \cdots & k_{1r_{2}} & \to & k_{1-}\\
			k_{21} & k_{22} & \cdots & k_{2r_{2}} & \to & k_{2-}\\
			\vdots & \vdots & \ddots & \vdots & & \vdots \\
			k_{r_{1}1} & k_{r_{1}2} & \cdots & k_{r_{1}r_{2}} & \to & k_{r_{1}-}\\
			\downarrow & \downarrow & & \downarrow & & \\
			k_{-1} & k_{-2} & \cdots & k_{-r_{2}} & & 
		\end{matrix}
	\end{align*}
	The set $\{ F_{K}(\bm{w})|K\in \mathbb{Z}_{\ge 0}^{r_{1}r_{2}}\}$ is a $\mathbb{C}$-linear basis of $S$. We define
	\begin{align*}
		\sigma:S\to L^{-1}R, F_{K}(\bm{w})\mapsto \frac{\prod\limits_{i=1}^{r_{1}}(u_{i})^{\underline{k_{i-}}}\prod\limits_{j=1}^{r_{2}}(v_{j})^{\underline{k_{-j}}}}{(s)^{\underline{|K|}}}.
	\end{align*}
	
	\begin{lemma}\label{l4.1}
		The map $\sigma$ is an $R$-module morphism and $\sigma|_{R}=id_{R}$.
	\end{lemma}
	\begin{proof}[Proof. ]
		The zero matrix $O$ corresponds to $F_{O}=1$, so the last claim follows from the first one. It suffices to show that $u_{l}\sigma(F_{K}(\bm{w}))=\sigma(u_{l}F_{K}(\bm{w}))$ and $v_{l^{\prime}}\sigma(F_{K}(\bm{w}))=\sigma(v_{l^{\prime}}F_{K}(\bm{w}))$ for all $K\in \mathbb{Z}_{\ge 0}^{r_{1}r_{2}}$, $1\le l\le r_{1}$, and $1\le l^{\prime}\le r_{2}$. We only prove $u_{l}\sigma(F_{K}(\bm{w}))=\sigma(u_{l}F_{K}(\bm{w}))$, since the other equality follows by a similar argument. Using the identity $z(z)^{\underline{k}}=(z)^{\underline{k+1}}+k(z)^{\underline{k}}$, one can see that 
		\begin{align*}
			u_{l}F_{K}(\bm{w}) & =w_{l1}\prod_{i,j}(w_{ij})^{\underline{k_{ij}}}+\cdots +w_{lr_{2}}\prod_{i,j}(w_{ij})^{\underline{k_{ij}}}\\
			& =F_{K+E_{l1}}(\bm{w})+k_{l1}F_{K}(\bm{w})\cdots +F_{K+E_{lr_{2}}}(\bm{w})+k_{lr_{2}}F_{K}(\bm{w})\\
			& =F_{K+E_{l1}}(\bm{w})+\cdots +F_{K+E_{lr_{2}}}(\bm{w})+k_{l-}F_{K}(\bm{w}),
		\end{align*}
		where $E_{ij}$ is the matrix whose $(i,j)$-entry is $1$ and all other entries are $0$. Hence $\sigma(u_{l}F_{K}(\bm{w}))=\sigma(F_{K+E_{l1}}(\bm{w}))+\cdots +\sigma(F_{K+E_{lr_{2}}}(\bm{w}))+k_{l-}\sigma(F_{K}(\bm{w}))$. By definition of $\sigma$, we have 
		\begin{align*}
			\sigma(F_{K+E_{l_{1}l_{2}}}(\bm{w}))=\sigma(F_{K}(\bm{w}))\frac{(u_{l_{1}}-k_{l_{1}-})(v_{l_{2}}-k_{-l_{2}})}{s-|K|}.
		\end{align*}
		It follows that $\sigma(u_{l}F_{K}(\bm{w}))=\bigg(k_{l-}+\frac{\sum\limits_{j=1}^{r_{2}}(u_{l}-k_{l-})(v_{j}-k_{-j})}{s-|K|}\bigg)\sigma(F_{K}(\bm{w}))=u_{l}\sigma(F_{K}(\bm{w}))$.
	\end{proof}
	
	\par Recall that we can write $A_{X,I}\otimes_{\mathbb{C}[s]} A_{Y,J}$ and $A_{X\times Y,IJ}$ as 
	\begin{center}
		$A_{X,I}\otimes_{\mathbb{C}[s]} A_{Y,J}=\mathcal{O}_{X\times Y}[\frac{1}{f_{1}\cdots f_{r_{1}}g_{1}\cdots g_{r_{2}}}]\otimes R$,\\
		$A_{X\times Y,IJ}=\mathcal{O}_{X\times Y}[\frac{1}{f_{1}\cdots f_{r_{1}}g_{1}\cdots g_{r_{2}}}]\otimes S$,
	\end{center}
	so $\sigma$ induces an $R$-module morphism $\lambda:A_{X\times Y,IJ}\to L^{-1}A_{X,I}\otimes_{\mathbb{C}[s]} A_{Y,J}$. Our next goal is to restrict $\lambda$ to $M_{X\times Y,IJ}$.
	
	\begin{lemma}\label{l4.2}
		Let $a,b\in \mathbb{Z}_{\ge 0}$. Then the equality
		\begin{align*}
			(z)^{\underline{a}}(z)^{\underline{b}}=\sum\limits_{k=0}^{\min\{ a,b\}}\binom{a}{k}\binom{b}{k}k!(z)^{\underline{a+b-k}}
		\end{align*}
		holds.
	\end{lemma}
	\begin{proof}[Proof. ]
		We can assume that $a\le b$ without loss of generality. By Chu-Vandermonde identity, we get
		\begin{align*}
			(z)^{\underline{a}}=\sum\limits_{k=0}^{a}\binom{a}{k}(b)^{\underline{k}}(z-b)^{\underline{a-k}}.
		\end{align*}
		After multiplying by $(z)^{\underline{b}}$, we obtain the desired equality.
	\end{proof}
	
	\begin{lemma}\label{l4.3}
		Let $\bm{\gamma}=(\gamma_{ij})\in \mathbb{Z}^{r_{1}r_{2}}$ be a matrix-valued multi-index satisfying $|\bm{\gamma}|=1$. Set $\alpha_{i}=\gamma_{i1}+\cdots +\gamma_{ir_{2}}$ and $\beta_{j}=\gamma_{1j}+\cdots +\gamma_{r_{1}j}$ for all $i,j$. Then we have
		\begin{align*}
			\sigma(\prod_{\gamma_{ij}<0}\binom{w_{ij}}{-\gamma_{ij}}\cdot S)\subset \prod_{\alpha_{i}<0}\binom{u_{i}}{-\alpha_{i}}\prod_{\beta_{j}<0}\binom{v_{j}}{-\beta_{j}}\cdot L^{-1}R.
		\end{align*}
	\end{lemma}
	\begin{proof}[Proof. ]
		Let $\eta_{ij}=\max\{ -\gamma_{ij},0\}$. We write $\prod\limits_{\gamma_{ij}<0}\binom{w_{ij}}{-\gamma_{ij}}\cdot S$ as $\prod\limits_{i,j}(w_{ij})^{\underline{\eta_{ij}}}\cdot S$. Then it suffices to show that
		\begin{align*}
			\sigma(F_{K}(\bm{w})\prod\limits_{\gamma_{ij}<0}\binom{w_{ij}}{-\gamma_{ij}})\in \prod_{\alpha_{i}<0}\binom{u_{i}}{-\alpha_{i}}\prod_{\beta_{j}<0}\binom{v_{j}}{-\beta_{j}}\cdot L^{-1}R
		\end{align*} 
		for all $K\in \mathbb{Z}^{r_{1}r_{2}}_{\ge 0}$. 
		\par By Lemma \ref{l4.2}, we have
		\begin{align*}
			\sigma(F_{K}(\bm{w})\prod\limits_{\gamma_{ij}<0}\binom{w_{ij}}{-\gamma_{ij}})=\prod\limits_{i,j}\sum\limits_{l_{ij}=0}^{\min\{ k_{ij},\eta_{ij}\}}\binom{k_{ij}}{l_{ij}}\binom{\eta_{ij}}{l_{ij}}l_{ij}!(w_{ij})^{\underline{k_{ij}+\eta_{ij}-l_{ij}}}.
		\end{align*}
		Note that $k_{ij}+\eta_{ij}-l_{ij}\ge \eta_{ij}$. Suppose $F_{K}(\bm{w})\prod\limits_{\gamma_{ij}<0}\binom{w_{ij}}{-\gamma_{ij}}=\sum\limits_{\mu=0}^{c}a_{\mu}F_{K^{\mu}}$ for $a_{\mu}\in \mathbb{C}$ and $K^{\mu}=(k^{\mu}_{ij})\in \mathbb{Z}_{\ge 0}^{r_{1}r_{2}}$. One can see that $(k^{\mu}_{ij})\ge \eta_{ij}$, so we have
		\begin{center}
			$(k^{\mu}_{i-})\ge \max \{-\alpha_{i},0\}$,\\
			$(k^{\mu}_{-j})\ge \max \{-\beta_{j},0\}$.
		\end{center}
		It follows that $\sigma(F_{K^{\mu}}(\bm{w}))\in \prod_{\alpha_{i}<0}\binom{u_{i}}{-\alpha_{i}}\prod_{\beta_{j}<0}\binom{v_{j}}{-\beta_{j}}\cdot L^{-1}R$.
	\end{proof}
	
	\begin{lemma}\label{l4.4}
		We have $\lambda(M_{X\times Y,IJ})\subset L^{-1}M_{I,J}$.
	\end{lemma}
	\begin{proof}[Proof. ]
		Suppose $F\in M_{X\times Y,IJ}$ satisfies 
		\begin{align*}
			Ff_{1}^{u_{1}}\cdots f_{r_{1}}^{u_{r_{1}}}g_{1}^{v_{1}}\cdots g_{r_{2}}^{v_{r_{2}}}=\sum_{|\bm{\gamma}|=1}P_{\bm{\gamma}}\cdot \bigg( \prod_{\gamma_{ij}<0}\binom{w_{ij}}{-\gamma_{ij}}\bigg)f_{1}^{u_{1}+\alpha_{1}}\cdots f_{r_{1}}^{u_{r_{1}}+\alpha_{r_{1}}}g_{1}^{v_{1}+\beta_{1}}\cdots g_{r_{2}}^{v_{r_{2}}+\beta_{r_{2}}},
		\end{align*}
		where $P_{\bm{\gamma}}\in \mathcal{D}_{X\times Y}[\bm{w}]$, $\alpha_{i}=\gamma_{i1}+\cdots +\gamma_{ir_{2}}$, and $\beta_{j}=\gamma_{1j}+\cdots +\gamma_{r_{1}j}$. We write $P_{\bm{\gamma}}$ as $\sum\limits_{\nu\in I_{\bm{\gamma}}}h_{\bm{\gamma},\nu}(\bm{w})Q_{\bm{\gamma},\nu}$, where $h_{\bm{\gamma},\nu}(\bm{w})\in S$ and $Q_{\bm{\gamma},\nu}\in \mathcal{D}_{X\times Y}$. 
		\par Suppose $Q_{\bm{\gamma},\nu}f_{1}^{u_{1}+\alpha_{1}}\cdots f_{r_{1}}^{u_{r_{1}}+\alpha_{r_{1}}}g_{1}^{v_{1}+\beta_{1}}\cdots g_{r_{2}}^{v_{r_{2}}+\beta_{r_{2}}}$ is of the form
		\begin{align*}
			\sum\limits_{\eta\in I_{\bm{\gamma},\nu}}\theta_{\bm{\gamma},\nu,\eta}G_{\bm{\gamma},\nu,\eta}f_{1}^{u_{1}}\cdots f_{r_{1}}^{u_{r_{1}}}g_{1}^{v_{1}}\cdots g_{r_{2}}^{v_{r_{2}}}.
		\end{align*}
		By Lemma \ref{l4.1}, $\lambda$ is an $R$-module morphism, so we have
		\begin{align*}
			\lambda(F) & =\sum\limits_{|\bm{\gamma}|=1}\sum\limits_{\nu\in I_{\bm{\gamma}}}\sum\limits_{\eta\in I_{\bm{\gamma},\nu}}\lambda(h_{\bm{\gamma},\nu}\theta_{\bm{\gamma},\nu,\eta}G_{\bm{\gamma},\nu,\eta}\prod_{\gamma_{ij}<0}\binom{w_{ij}}{-\gamma_{ij}})\\
			& =\sum\limits_{|\bm{\gamma}|=1}\sum\limits_{\nu\in I_{\bm{\gamma}}}\sum\limits_{\eta\in I_{\bm{\gamma},\nu}}\theta_{\bm{\gamma},\nu,\eta}G_{\bm{\gamma},\nu,\eta}\lambda(h_{\bm{\gamma},\nu}\prod_{\gamma_{ij}<0}\binom{w_{ij}}{-\gamma_{ij}}).
		\end{align*}
		By Lemma \ref{l4.3}, it is in $L^{-1}M_{I,J}$.
	\end{proof}
	
	\begin{proof}[Proof of $b_{IJ}(s)=b_{I,J}(s)$. ]
		We prove $b_{I,J}(s)|b_{IJ}(s)$. By Lemma \ref{l4.1} and Lemma \ref{l4.4}, we have $b_{IJ}(s)=\sigma(b_{IJ}(s))\in L^{-1}M_{I,J}$. The polynomial $b_{IJ}(s)$ is also in $\mathbb{C}[s]$. Hence there exists $a_{1},\cdots ,a_{N}\in \mathbb{Z}_{\ge 0}$ such that $(s-a_{1})\cdots (s-a_{N})b_{IJ}(s)\in \mathbb{C}[s]\cap M_{I,J}=b_{I,J}(s)\cdot \mathbb{C}[s]$. It follows that $b_{I,J}(s)|(s-a_{1})\cdots (s-a_{N})b_{IJ}(s)$. All roots of $b_{I,J}(s)$ are negative, so $b_{I,J}(s)|b_{IJ}(s)$.
	\end{proof}
	
	\par The proof of Theorem \ref{t1.1} is complete.

	\section{An alternative proof of $b_{I,J}(s)=b_{I}(s)b_{J}(s)$}\label{s5}
	
	\par Let $F_{I}=f_{1}X_{1}+\cdots +f_{r_{1}}X_{r_{1}}\in \mathbb{C}[\bm{x},\bm{X}]$ and $G_{J}=g_{1}Y_{1}+\cdots +g_{r_{2}}Y_{r_{2}}\in \mathbb{C}[\bm{y},\bm{Y}]$. By Theorem \ref{c1.1} and Theorem \ref{t2.2}, $b_{I}(s)b_{J}(s)=\frac{1}{(s+1)^{2}}\cdot b_{F_{I}}(s)b_{G_{J}}(s)=\frac{1}{(s+1)^{2}}b_{F_{I}G_{J}}(s)$. In this section, we prove $b_{I,J}(s)=\frac{1}{(s+1)^{2}}b_{F_{I}G_{J}}(s)$ via a similar argument to that in \cite{c6}. Then $b_{I,J}(s)=b_{I}(s)b_{J}(s)$ is an immediate consequence.
	\par We shall also need some further notation for multi-indices. Let $\bm{\alpha}=(\alpha_{1},\cdots ,\alpha_{r})\in \mathbb{Z}_{\ge 0}^{r}$ be a multi-index. We write $\bm{\alpha}=\alpha_{1}!\cdots \alpha_{r}!$ and $\binom{|\bm{\alpha}|}{\bm{\alpha}}=\frac{|\bm{\alpha}|!}{\bm{\alpha}!}$.
	
	\begin{proof}[Proof of $b_{I,J}(s)=b_{I}(s)b_{J}(s)$. ]
		We prove $b_{I,J}(s)=\frac{1}{(s+1)^{2}}b_{F_{I}G_{J}}(s)$. Suppose that
		\begin{equation}\label{e7}
			b_{F_{I}G_{J}}(s)(F_{I}G_{J})^{s}=P\cdot (F_{I}G_{J})^{s+1}.
		\end{equation}
		We write $P$ as
		\begin{align*}
			P=\sum_{\substack{\bm{\alpha},\bm{\alpha}^{\prime}\in \mathbb{Z}_{\ge 0}^{r_{1}}\\
			\bm{\beta},\bm{\beta}^{\prime}\in \mathbb{Z}_{\ge 0}^{r_{2}}}}\bigg( \frac{1}{\bm{\alpha}^{\prime}!}P_{\bm{\alpha},\bm{\alpha}^{\prime}}\bm{X}^{\bm{\alpha}}\partial_{\bm{X}}^{\bm{\alpha}^{\prime}}\bigg)\bigg( \frac{1}{\bm{\beta}^{\prime}!}Q_{\bm{\beta},\bm{\beta}^{\prime}}\bm{Y}^{\bm{\beta}}\partial_{\bm{Y}}^{\bm{\beta}^{\prime}}\bigg),
		\end{align*}
		where $P_{\bm{\alpha},\bm{\alpha}^{\prime}}\in \mathcal{D}_{X}[s]$ and $Q_{\bm{\beta},\bm{\beta}^{\prime}}\in \mathcal{D}_{Y}[s]$. For $k,l\in \mathbb{Z}_{\ge 0}$, recall the set $T_{k,l}=\{ \bm{a}\in \mathbb{Z}_{\ge 0}^{k}\;|\;|\bm{a}|=l\}$ defined in the proof of Lemma \ref{l3.2}. For all $l\in \mathbb{Z}_{\ge 0}$, taking $s=l$ in \eqref{e7} yields
		\begin{align*}
			(F_{I}G_{J})^{l}=\sum\limits_{\substack{\bm{a}\in T_{r_{1},l}\\
			\bm{b}\in T_{r_{2},l}}}\binom{l}{\bm{a}}\binom{l}{\bm{b}}f_{1}^{a_{1}}\cdots f_{r_{1}}^{a_{r_{1}}}g_{1}^{b_{1}}\cdots g_{r_{2}}^{b_{r_{2}}}\bm{X}^{\bm{a}}\bm{Y}^{\bm{b}},
		\end{align*} 
		\begin{align*}
			P(l)\cdot (F_{I}G_{J})^{l+1}=\sum\limits_{\substack{\bm{\alpha},\bm{\alpha}^{\prime}\in \mathbb{Z}_{\ge 0}^{r_{1}}\\
			\bm{\beta},\bm{\beta}^{\prime}\in \mathbb{Z}_{\ge 0}^{r_{2}}}}\sum\limits_{\substack{\bm{a}\in T_{r_{1},l+1}\\
			\bm{b}\in T_{r_{2},l+1}}}A_{\bm{\alpha},\bm{\alpha}^{\prime},\bm{a}}(l)B_{\bm{\beta},\bm{\beta}^{\prime},\bm{b}}(l)\bm{X}^{\bm{a}-\bm{\alpha}^{\prime}+\bm{\alpha}}\bm{Y}^{\bm{b}-\bm{\beta}^{\prime}+\bm{\beta}},
		\end{align*}
		where $A_{\bm{\alpha},\bm{\alpha}^{\prime},\bm{a}}(l)$ and $B_{\bm{\beta},\bm{\beta}^{\prime},\bm{b}}(l)$ are
		\begin{align*}
			A_{\bm{\alpha},\bm{\alpha}^{\prime},\bm{a}}(l)=\binom{l+1}{\bm{a}}P_{\bm{\alpha},\bm{\alpha}^{\prime}}(l)\cdot \prod\limits_{i=1}^{r_{1}}f_{i}^{a_{i}}\binom{a_{i}}{\alpha_{i}^{\prime}},\\
			B_{\bm{\beta},\bm{\beta}^{\prime},\bm{b}}(l)=\binom{l+1}{\bm{b}}Q_{\bm{\beta},\bm{\beta}^{\prime}}(l)\cdot \prod\limits_{j=1}^{r_{2}}g_{j}^{b_{j}}\binom{b_{j}}{\beta_{j}^{\prime}}.
		\end{align*}
		Fixing $\bm{a}\in T_{r_{1},l}$ and $\bm{b}\in T_{r_{2},l}$, the coefficient of $\bm{X}^{\bm{a}}\bm{Y}^{\bm{b}}$ in $P(l)\cdot (F_{I}G_{J})^{l+1}$ is
		\begin{align*}
			\sum\limits_{\substack{\bm{\alpha},\bm{\alpha}^{\prime}\in \mathbb{Z}_{\ge 0}^{r_{1}}\\
			\bm{a}+\bm{\alpha}^{\prime}-\bm{\alpha}\in T_{r_{1},l}}}\sum\limits_{\substack{\bm{\beta},\bm{\beta}^{\prime}\in \mathbb{Z}_{\ge 0}^{r_{2}}\\
			\bm{b}+\bm{\beta}^{\prime}-\bm{\beta}\in T_{r_{2},l}}}A_{\bm{\alpha},\bm{\alpha}^{\prime},\bm{a}+\bm{\alpha}^{\prime}-\bm{\alpha}}(l)B_{\bm{\beta},\bm{\beta}^{\prime},\bm{b}+\bm{\beta}^{\prime}-\bm{\beta}}(l).
		\end{align*}
		From now on, we omit the conditions $\bm{\alpha},\bm{\alpha}^{\prime}\in \mathbb{Z}_{\ge 0}^{r_{1}}$ and $\bm{\beta},\bm{\beta}^{\prime}\in \mathbb{Z}_{\ge 0}^{r_{2}}$ from the notation. By comparing the coefficients of $\bm{X}^{\bm{a}}\bm{Y}^{\bm{b}}$, the equality \eqref{e7} shows 
		\begin{equation}\label{e8}
			b_{F_{I}G_{J}}(l)\binom{l}{\bm{a}}\binom{l}{\bm{b}}f_{1}^{a_{1}}\cdots f_{r_{1}}^{a_{r_{1}}}g_{1}^{b_{1}}\cdots g_{r_{2}}^{b_{r_{2}}}=\sum\limits_{\substack{\bm{a}+\bm{\alpha}^{\prime}-\bm{\alpha}\in T_{r_{1},l+1}\\
			\bm{b}+\bm{\beta}^{\prime}-\bm{\beta}\in T_{r_{2},l+1}}}A_{\bm{\alpha},\bm{\alpha}^{\prime},\bm{a}+\bm{\alpha}^{\prime}-\bm{\alpha}}(l)B_{\bm{\beta},\bm{\beta}^{\prime},\bm{b}+\bm{\beta}^{\prime}-\bm{\beta}}(l).
		\end{equation} 
		One can see that $b_{F_{I}G_{J}}(s)$ is the monic polynomial of the smallest degree such that there exist $P_{\bm{\alpha},\bm{\alpha}^{\prime}}$ and $Q_{\bm{\beta},\bm{\beta}^{\prime}}$ admitting an equality with the same structure as \eqref{e8} after taking $s=l$ for all $l\in \mathbb{Z}_{\ge 0}$. Applying the identities
		\begin{align*}
			\binom{l+1}{\bm{a}+\bm{\alpha}^{\prime}-\bm{\alpha}}\prod\limits_{i=1}^{r_{1}}\binom{a_{i}+\alpha_{i}^{\prime}-\alpha_{i}}{\alpha_{i}^{\prime}}=(l+1)\binom{l}{\bm{a}}\prod\limits_{i=1}^{r_{1}}\frac{a_{i}!}{\alpha_{i}^{\prime}!(a_{i}-\alpha_{i})!},
		\end{align*} 
		\begin{align*}
			\binom{l+1}{\bm{b}+\bm{\beta}^{\prime}-\bm{\beta}}\prod\limits_{j=1}^{r_{2}}\binom{b_{j}+\beta_{j}^{\prime}-\beta_{j}}{\beta_{j}^{\prime}}=(l+1)\binom{l}{\bm{b}}\prod\limits_{j=1}^{r_{2}}\frac{b_{j}!}{\beta_{j}^{\prime}!(b_{j}-\beta_{j})!},
		\end{align*} 
		the equality \eqref{e8} shows that $\frac{b_{F_{I}G_{J}}(l)}{(l+1)^{2}}f_{1}^{a_{1}}\cdots f_{r_{1}}^{a_{r_{1}}}g_{1}^{b_{1}}\cdots g_{r_{2}}^{b_{r_{2}}}$ is equal to
		\begin{equation}\label{e9}
			\sum\limits_{\substack{\bm{a}+\bm{\alpha}^{\prime}-\bm{\alpha}\in T_{r_{1},l}\\
			\bm{b}+\bm{\beta}^{\prime}-\bm{\beta}\in T_{r_{2},l}}}P_{\bm{\alpha},\bm{\alpha}^{\prime}}(l)Q_{\bm{\beta},\bm{\beta}^{\prime}}(l)\cdot \bigg( \prod\limits_{i=1}^{r_{1}}\frac{\alpha_{i}!}{\alpha_{i}^{\prime}!}\binom{a_{i}}{\alpha_{i}}\prod\limits_{j=1}^{r_{2}}\frac{\beta_{j}!}{\beta_{j}^{\prime}!}\binom{b_{j}}{\beta_{j}}\bigg)\prod\limits_{i=1}^{r_{1}}f_{i}^{a_{i}+\alpha_{i}^{\prime}-\alpha_{i}}\prod\limits_{j=1}^{r_{2}}g_{j}^{b_{j}+\beta^{\prime}_{j}-\beta_{j}}.
		\end{equation}
		\par The equality that $\frac{b_{F_{I}G_{J}}(l)}{(l+1)^{2}}f_{1}^{a_{1}}\cdots f_{r_{1}}^{a_{r_{1}}}g_{1}^{b_{1}}\cdots g_{r_{2}}^{b_{r_{2}}}$ equals \eqref{e9} holds for all $\bm{a}\in \mathbb{Z}_{\ge 0}^{r_{1}}$ and $\bm{b}\in \mathbb{Z}_{\ge 0}^{r_{2}}$. In the ring $R=\mathbb{C}[\bm{u},\bm{v}]/(|\bm{u}|-|\bm{v}|)$, we can replace $\bm{a}$, $\bm{b}$ by $\bm{u}$, $\bm{v}$ to obtain
		\begin{align*}
			& \frac{b_{F_{I}G_{J}}(s)}{(s+1)^{2}}\prod\limits_{i=1}^{r_{1}}f_{i}^{u_{i}}\prod\limits_{j=1}^{r_{2}}g_{j}^{v_{j}}\\
			= & \sum\limits_{\substack{|\bm{\alpha}^{\prime}-\bm{\alpha}|=1\\
			|\bm{\beta}^{\prime}-\bm{\beta}|=1}}P_{\bm{\alpha},\bm{\alpha}^{\prime}}(s)Q_{\bm{\beta},\bm{\beta}^{\prime}}(s)\cdot \bigg( \prod\limits_{i=1}^{r_{1}}\frac{\alpha_{i}!}{\alpha_{i}^{\prime}!}\binom{u_{i}}{\alpha_{i}}\prod\limits_{j=1}^{r_{2}}\frac{\beta_{j}!}{\beta_{j}^{\prime}!}\binom{v_{j}}{\beta_{j}}\bigg)\prod\limits_{i=1}^{r_{1}}f_{i}^{u_{i}+\alpha_{i}^{\prime}-\alpha_{i}}\prod\limits_{j=1}^{r_{2}}g_{j}^{v_{j}+\beta^{\prime}_{j}-\beta_{j}}\\
		    \in & \sum\limits_{\substack{|\bm{\mu}|=|\bm{\nu}|=1}}\sum\limits_{\substack{\bm{\alpha}\in \mathbb{Z}_{\ge 0}^{r_{1}},\bm{\alpha}+\bm{\mu}\in \mathbb{Z}_{\ge 0}^{r_{1}}\\
		    \bm{\beta}\in \mathbb{Z}_{\ge 0}^{r_{2}},\bm{\beta}+\bm{\nu}\in \mathbb{Z}_{\ge 0}^{r_{2}}}}\mathcal{D}_{X\times Y}[s]\cdot \bigg( \prod\limits_{i=1}^{r_{1}}\binom{u_{i}}{\alpha_{i}}\prod\limits_{j=1}^{r_{2}}\binom{v_{j}}{\beta_{j}}\bigg)\prod\limits_{i=1}^{r_{1}}f_{i}^{u_{i}+\mu_{i}}\prod\limits_{j=1}^{r_{2}}g_{j}^{v_{j}+\nu_{j}}.
		\end{align*}
		When $\mu_{i}\ge 0$, the ideal $(\binom{u_{i}}{\alpha_{i}}|\alpha_{i}+\mu_{i}\ge 0)\subset \mathbb{C}[u_{i}]$ is equal to $\mathbb{C}[u_{i}]$. When $\mu_{i}<0$, the ideal $(\binom{u_{i}}{\alpha_{i}}|\alpha_{i}+\mu_{i}\ge 0)\subset \mathbb{C}[u_{i}]$ is equal to $(\binom{u_{i}}{-\mu_{i}})$. A similar argument works for $\beta_{j}$ and $\nu_{j}$. Therefore, $\frac{b_{F_{I}G_{J}}(s)}{(s+1)^{2}}$ is the monic polynomial of the smallest degree such that
		\begin{align*}
			\frac{b_{F_{I}G_{J}}(s)}{(s+1)^{2}}\prod\limits_{i=1}^{r_{1}}f_{i}^{u_{i}}\prod\limits_{j=1}^{r_{2}}g_{j}^{v_{j}}\in \sum\limits_{\substack{|\bm{\mu}|=|\bm{\nu}|=1}}\mathcal{D}_{X\times Y}[s]\cdot \bigg( \prod\limits_{\mu_{i}<0}\binom{u_{i}}{-\mu_{i}}\prod\limits_{\nu_{j}<0}\binom{v_{j}}{-\nu_{j}}\bigg)\prod\limits_{i=1}^{r_{1}}f_{i}^{u_{i}+\mu_{i}}\prod\limits_{j=1}^{r_{2}}g_{j}^{v_{j}+\nu_{j}}.
		\end{align*}
		It follows that $b_{I,J}(s)=\frac{b_{F_{I}G_{J}}(s)}{(s+1)^{2}}$.
	\end{proof}
	
\section{Acknowledgement}
	
	On behalf of all authors, the corresponding author states that there is no conflict of interest.
	We thank Quan Shi for helpful discussions. H. Zuo acknowledges support from NSFC (grant No. 12671056) and BJNSF (grant No. 1252009).
	
	\end{document}